\documentclass{article}
\usepackage{amsthm}

\usepackage{xspace}
\usepackage{tikz}
\usepackage{amsmath}
\usepackage{booktabs}

\usepackage[numbers,square,sort&compress]{natbib}

\usepackage[a4paper,hscale=0.76,vscale=0.83]{geometry}
\usepackage{url} 
\usepackage{hyperref}

\newtheorem{theorem}{Theorem}

\newtheorem{lemma}[theorem]{Lemma}

\newcommand{\ie}{i.e.\xspace}

\title{From edge-disjoint paths to independent paths}

\author{Serge Gaspers\footnote{%
  Institute of Information Systems,
  Vienna University of Technology,
  \texttt{gaspers@kr.tuwien.ac.at}\newline
  Research supported by the European Research Council
  (COMPLEX REASON, 239962).
}}

\date{}

\begin{document}
 \maketitle

\begin{abstract}
Let $f(k)$ denote the maximum such that every simple undirected graph containing two vertices $s,t$ and $k$ edge-disjoint
$s$--$t$ paths, also contains two vertices $u,v$ and $f(k)$ independent $u$--$v$ paths. Here, a set of paths
is \emph{independent} if none of them contains an interior vertex of another. We prove that
\begin{align*}
 f(k) = \begin{cases}
         k & \text{ if $k \le 2$, and}\\
         3 & \text{ otherwise.}
        \end{cases}
\end{align*}
\end{abstract}

\bigskip

\noindent
Since independent paths are edge-disjoint, it is clear that $f(k)\le k$ for every positive integer $k$.

Let $\mathcal{P}$ be a set of edge-disjoint $s$--$t$ paths in a graph $G$. Clearly, if $|\mathcal{P}|\le 1$, then
the paths in $\mathcal{P}$ are independent. If $\mathcal{P}=\{ P_1,P_2 \}$, a set of two independent $u$--$v$ paths can
easily be obtained as follows. Set $u:=s$ and let $v$ be the vertex that belongs to both $P_1$ and $P_2$ and is closest to $s$ on $P_1$.
Then, the $u$--$v$ subpaths of $P_1$ and $P_2$ are independent. This proves that $f(k)=k$ if $k\le 2$.

\smallskip

The lower bound for $f(k), k\ge 3$, is provided by the following lemma.

\begin{lemma}\label{lem:lb}
 Let $G=(V,E)$ be a graph. If there are two vertices $s,t\in V$ with 3
edge-disjoint $s$--$t$ paths in $G$,
  then there are two vertices $u,v\in V$ with 3 independent $u$--$v$
paths in $G$.
\end{lemma}
\begin{proof}
Let $P_1,P_2,P_3$ denote 3 edge-disjoint $s$--$t$ paths,
and let $S=\{s_1,s_2,s_3\}$, where $s_i$ neighbors $s$ on $P_i$, $1\le i\le 3$.
Consider the connected component $G'$ of $G\setminus \{s\}$ containing $t$.
Then, $G'$ contains all vertices from $S$.
Let $T$ be a spanning tree of $G'$.
Select $v$ such that the $s_i$--$v$ subpaths of $T$, $1\le i\le 3$, are independent.
This vertex $v$ belongs to every subpath of $T$ that has
two vertices from $S$ as endpoints.
To see that this vertex exists, consider the $s_1$--$s_3$ subpath $P_{1,3}$ of $T$
and the $s_2$--$s_3$ subpath $P_{2,3}$ of $T$.
Set $v$ to be the vertex that belongs to both $P_{1,3}$ and $P_{2,3}$ and is closest to $s_2$ on $P_{2,3}$
(if $P_{1,3}$ contains $s_2$, then $v=s_2$).
Set $u:=s$, and obtain 3 independent $u$--$v$ paths in $G$ by moving
from $u$ to $s_i$, and then along the $s_i$--$v$ subpath of $T$ to $v$,
$1\le i\le 3$.
\end{proof}

For the upper bound, consider the following family of graphs, the \emph{recursive diamond graphs}~\cite{GuptaNRS04}.
The recursive diamond graph of order $0$ is $G_0=(\{s,t\}, \{st\})$,
and the diamond graph $G_p$ of order $p\ge 1$ is obtained from $G_{p-1}$
by replacing each edge $e=xy$ by the set of edges $\{x p_e, p_e y, x q_e, q_e y\}$, where
$p_e$ and $q_e$ are new vertices.
See Figure \ref{fig} for an illustration.

\tikzset{var/.style={inner sep=.075em,circle,fill=black,draw},
         clause/.style={minimum size=1mm,rectangle,fill=white,draw},
         label distance=-1pt}

\begin{figure}[t]
 \centering
  \begin{tikzpicture}[scale=0.5]
   \node (s) at (0,0) [var,label=below:$s$] {};
   \node (t) at (0,4) [var,label=above:$t$] {};
   \draw (s)--(t);

   \node at (-0.5,2) {$G_0$};
  \end{tikzpicture}\hspace{0.5cm}
  \begin{tikzpicture}[scale=0.75]
   \node (s) at (0,0) [var,label=below:$s$] {};
   \node (t) at (0,4) [var,label=above:$t$] {};
   \node (p) at (-2,2) [var] {};
   \node (q) at (2,2) [var] {};

   \draw (s)--(p)--(t) (s)--(q)--(t);
   \node at (0,2) {$G_1$};
  \end{tikzpicture}\hspace{0.5cm}
  \begin{tikzpicture}[scale=0.875]
   \node (s) at (0,0) [var,label=below:$s$] {};
   \node (t) at (0,4) [var,label=above:$t$] {};
   \node (p) at (-2,2) [var] {};
   \node (q) at (2,2) [var] {};
   \node (p1) at (-0.7,1.3) [var] {};
   \node (p2) at (-1.3,0.7) [var] {};
   \node (p3) at (-0.7,2.7) [var] {};
   \node (p4) at (-1.3,3.3) [var] {};
   \node (q1) at (0.7,1.3) [var] {};
   \node (q2) at (1.3,0.7) [var] {};
   \node (q3) at (0.7,2.7) [var] {};
   \node (q4) at (1.3,3.3) [var] {};

   \draw (s)--(p1)--(p) (s)--(p2)--(p)--(p3)--(t) (p)--(p4)--(t)
   (s)--(q1)--(q) (s)--(q2)--(q)--(q3)--(t) (q)--(q4)--(t);
   \node at (0,2) {$G_2$};
  \end{tikzpicture}\hspace{0.5cm}
  \begin{tikzpicture}
   \node (s) at (0,0) [var,label=below:$s$] {};
   \node (t) at (0,4) [var,label=above:$t$] {};
   \node (p) at (-2,2) [var] {};
   \node (q) at (2,2) [var] {};
   \node (p1) at (-0.7,1.3) [var] {};
   \node (p2) at (-1.3,0.7) [var] {};
   \node (p3) at (-0.7,2.7) [var] {};
   \node (p4) at (-1.3,3.3) [var] {};
   \node (q1) at (0.7,1.3) [var] {};
   \node (q2) at (1.3,0.7) [var] {};
   \node (q3) at (0.7,2.7) [var] {};
   \node (q4) at (1.3,3.3) [var] {};

   \node (q11) at (0.25,0.7) [var] {};
   \node (q12) at (0.45,0.6) [var] {};
   \node (q13) at (1.3,1.75) [var] {};
   \node (q14) at (1.4,1.55) [var] {};
   \node (q21) at (0.6,0.45) [var] {};
   \node (q22) at (0.7,0.25) [var] {};
   \node (q23) at (1.55,1.4) [var] {};
   \node (q24) at (1.75,1.3) [var] {};

   \node (q31) at (0.25,3.3) [var] {};
   \node (q32) at (0.45,3.4) [var] {};
   \node (q33) at (1.3,2.25) [var] {};
   \node (q34) at (1.4,2.45) [var] {};
   \node (q41) at (0.6,3.55) [var] {};
   \node (q42) at (0.7,3.75) [var] {};
   \node (q43) at (1.55,2.6) [var] {};
   \node (q44) at (1.75,2.7) [var] {};

   \node (p11) at (-0.25,0.7) [var] {};
   \node (p12) at (-0.45,0.6) [var] {};
   \node (p13) at (-1.3,1.75) [var] {};
   \node (p14) at (-1.4,1.55) [var] {};
   \node (p21) at (-0.6,0.45) [var] {};
   \node (p22) at (-0.7,0.25) [var] {};
   \node (p23) at (-1.55,1.4) [var] {};
   \node (p24) at (-1.75,1.3) [var] {};

   \node (p31) at (-0.25,3.3) [var] {};
   \node (p32) at (-0.45,3.4) [var] {};
   \node (p33) at (-1.3,2.25) [var] {};
   \node (p34) at (-1.4,2.45) [var] {};
   \node (p41) at (-0.6,3.55) [var] {};
   \node (p42) at (-0.7,3.75) [var] {};
   \node (p43) at (-1.55,2.6) [var] {};
   \node (p44) at (-1.75,2.7) [var] {};

   \draw (s)--(p11)--(p1) (s)--(p12)--(p1)--(p13)--(p) (p1)--(p14)--(p) (s)--(p21)--(p2)
   (s)--(p22)--(p2)--(p23)--(p) (p2)--(p24)--(p)--(p33)--(p3) (p)--(p34)--(p3)--(p31)--(t)
   (p3)--(p32)--(t) (p)--(p43)--(p4) (p)--(p44)--(p4)--(p41)--(t) (p4)--(p42)--(t)
   (s)--(q11)--(q1) (s)--(q12)--(q1)--(q13)--(q) (q1)--(q14)--(q) (s)--(q21)--(q2)
   (s)--(q22)--(q2)--(q23)--(q) (q2)--(q24)--(q)--(q33)--(q3) (q)--(q34)--(q3)--(q31)--(t)
   (q3)--(q32)--(t) (q)--(q43)--(q4) (q)--(q44)--(q4)--(q41)--(t) (q4)--(q42)--(t);
   \node at (0,2) {$G_3$};
  \end{tikzpicture}
 \caption{The recursive diamond graphs of order $0,1,2$, and $3$.}
 \label{fig}
\end{figure}
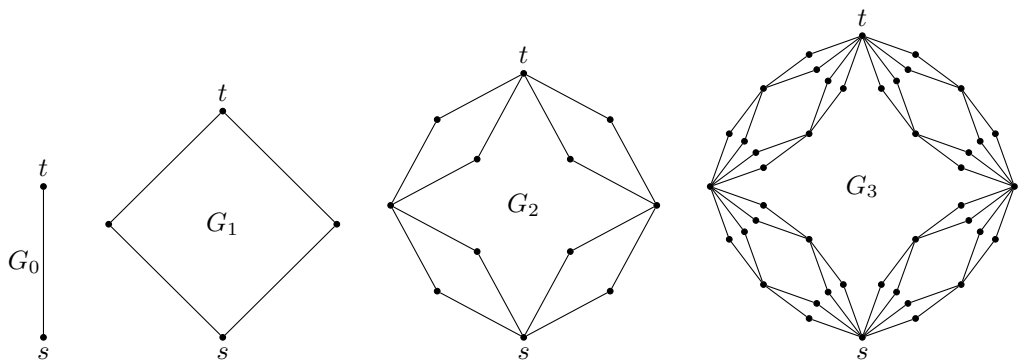

The following lemma entails the upper bound for $f(k), k\ge 3$.

\begin{lemma}\label{lem:ub}
 For every $k\ge 3$, there is a graph $G=(V,E)$ containing two vertices $s,t\in V$ with $k$
 edge-disjoint $s$--$t$ paths,
 but no two vertices $u,v\in V$ with 4 independent $u$--$v$
 paths.
\end{lemma}
\begin{proof}
Consider the diamond graph
$G=G_p$ of order $p = \lceil \log k \rceil$.
$G$ has $2^p \ge k$ edge-disjoint $s$--$t$ paths.
Let $u,v$ be any two vertices in $G$. We will show that there are at
most 3 independent $u$--$v$ paths.

Observe that each recursive diamond graph $G_r$ contains $4$ edge-disjoint copies of $G_{r-1}$.
The \emph{extremities} of $G_r$ are the vertices $s$ and $t$, and the \emph{extremities} of a subgraph $H$ of $G_r$
that is isomorphic to $G_{r'}, r'<r$, are the two vertices from $H$ whose neighborhoods in $G_r$ are not a subset of $V(H)$.

Let $Q$ be the smallest vertex set containing $u$ and $v$ such that
$G[Q]$ is a recursive diamond graph. Let $q$ be the order of the recursive diamond graph $G[Q]$.

If $q=0$, then $u v$ is an edge in $G$, and either $u$ or $v$ has degree $2$.
But then, the number of independent $u$--$v$ paths in $G$ is at most 2 since
independent paths pass through distinct neighbors of $u$ and $v$.

If $q>0$, then $u v$ is not an edge in $G$.
Decompose $G[Q]$ into 4 edge-disjoint graphs $H_1, \dots, H_4$
isomorphic to $G_{q-1}$
such that $u\in V(H_1)$ and the $H_i$ are ordered cyclically by their index
(\ie, $V(H_1) \cap V(H_3) = \emptyset$).
Since we chose $Q$ to be minimum, $u$ and $v$ do not belong to the same $H_i, 1\le i\le 4$.
If $u \notin V(H_2) \cup V(H_4)$, then
the extremities of $H_1$ are a $u$--$v$-vertex cut of size $2$ in $G[Q]$ and in $G$.
Otherwise, suppose, without loss of generality, that $u\in V(H_1) \cap V(H_2)$.
Since $v \notin V(H_1) \cup V(H_2)$, the other two extremities of $H_1$ and $H_2$ form a
$u$--$v$-vertex cut $C$ of size $2$ in $G[Q]$. The set $C$ is also a $u$--$v$-vertex cut in $G$,
unless $q<p$ and $u$ is an extremity of another subgraph $J$ of $G$ isomorphic to $G_q$ that is edge-disjoint from $G[Q]$.
In the latter case, add the other extremity of $J$ to $C$ to obtain a $u$--$v$-vertex cut in $G$ of size~$3$.


Since $G$ has a $u$--$v$-vertex cut of size at most $3$, by Menger's theorem \cite{Menger27}, there are
at most 3 independent $u$--$v$ paths in $G$.
\end{proof}

\paragraph{An application}
Lemma \ref{lem:lb} has been used in an algorithm \cite{GaspersS12} for the detection of backdoor sets
to ease Satisfiability solving. A backdoor set of a propositional formula is a set of variables
such that assigning truth values to the variables in the backdoor set moves the formula into a
polynomial-time decidable class; see \cite{GaspersSzeider11festschrift} for a survey.
The class of nested formulas was introduced by Knuth \cite{Knuth90} and their satisfiability can be decided
in polynomial time.
To find a backdoor set to the class of nested formulas, the algorithm from \cite{GaspersS12}
considers the clause-variable incidence graph of the formula.
If the formula is nested, this graph does not contain a $K_{2,3}$-minor with the additional property that
the independent set of size $3$ is obtained by contracting $3$ connected subgraphs containing a variable each.
In the correctness proof of the algorithm it is shown that in certain cases the formula does not have a small backdoor set.
This is shown by exhibiting
two vertices $u,v$ and 3 independent $u$--$v$ paths in an auxiliary graph using Lemma \ref{lem:lb}.
Expanding these edges to the paths they represent in the formula's incident graph
gives rise to a $K_{2,3}$-minor with the desired property.

On the other hand, Lemma \ref{lem:ub} shows the limitations
of this approach if we would like to enlarge the target class to more general formulas.

\paragraph{Acknowledgment}
We thank Chandra Chekuri for bringing the recursive diamond graphs to our attention \cite{Gaspers12SE},
and we thank Herbert Fleischner for valuable discussions on an earlier version of this note.

\bibliographystyle{plainnat}
\bibliography{literature}

\end{document}